\documentclass{amsart}

\usepackage[all]{xypic}
\usepackage{epsf}
\usepackage[centertags]{amsmath}
\usepackage{amsfonts}
\usepackage{amssymb}
\usepackage{amsthm}
\usepackage{newlfont}

 \newtheorem{thm}{Theorem}[section]

 \newtheorem{prop}[thm]{Proposition}
 \theoremstyle{definition}
 \newtheorem{defn}[thm]{Definition}
 \theoremstyle{remark}
 \newtheorem{rem}[thm]{Remark}
 \theoremstyle{remark}
 
 \theoremstyle{definition}
 \newtheorem{notn}[thm]{Notation}
 \numberwithin{equation}{section}

 \font \cyr=wncyr10
 \define\Sh{\hbox{\cyr Sh}}

 \newcommand{\an}{\mathrm{an}}

 \newcommand{\Spec}{\mathrm{Spec}}
 \newcommand{\Frob}{\mathrm{Frob}}
 \newcommand{\Spf}{\mathrm{Spf}}
 
 \newcommand{\Ver}{\mathrm{Ver}}
 \newcommand{\Ed}{\mathrm{Ed}}
 \newcommand{\Aut}{\mathrm{Aut}}
 
 \newcommand{\Supp}{\mathrm{Supp}}
 \newcommand{\Pic}{\mathrm{Pic}}
 
 \newcommand{\ord}{\mathrm{ord}}
 \newcommand{\alg}{\mathrm{alg}}

 \newcommand{\Gal}{\mathrm{Gal}}
 \newcommand{\GL}{\mathrm{GL}}
 
 \newcommand{\PGL}{\mathrm{PGL}}
 \newcommand{\Div}{\mathrm{Div}}
 
 \newcommand{\sep}{\mathrm{sep}}

 \newcommand{\un}{\mathrm{ur}}

 \newcommand{\Stab}{\mathrm{Stab}}
 \newcommand{\Fix}{\mathrm{Fix}}

 \newcommand{\Odd}{\mathrm{Odd}}
 
 \renewcommand{\mod}{\mathrm{mod}}
 \newcommand{\Nr}{\mathrm{Nr}}
 
 \newcommand{\fr}{\mathfrak r}
 
 \newcommand{\fp}{\mathfrak p}
 \newcommand{\fn}{\mathfrak n}
 \newcommand{\fm}{\mathfrak m}

 \newcommand{\fD}{\mathfrak D}
 \newcommand{\fP}{\mathfrak P}
 \newcommand{\cO}{\mathcal{O}}

 \newcommand{\cC}{\mathcal{C}}
 
 \renewcommand{\cD}{\mathcal{D}}
 
 \newcommand{\cG}{\mathcal{G}}

 \newcommand{\cE}{\mathcal{E}}

 \newcommand{\cJ}{\mathcal{J}}
 \newcommand{\cI}{\mathcal{I}}
 \renewcommand{\cL}{\mathcal{L}}
 
 \newcommand{\cT}{\mathcal{T}}

 \newcommand{\C}{\mathbb{C}}
 \newcommand{\F}{\mathbb{F}}
 \newcommand{\Q}{\mathbb{Q}}
 \newcommand{\Z}{\mathbb{Z}}

 \newcommand{\M}{\mathbb{M}}
 \newcommand{\N}{\mathbb{N}}
 \newcommand{\p}{\mathbb{P}}
 
 \newcommand{\T}{\mathbb{T}}

 \newcommand{\bD}{\bar{D}}

 \newcommand{\bs}{\setminus}
 
 \newcommand{\Fi}{F_\infty}

 \newcommand{\G}{\Gamma}
 \newcommand{\La}{\Lambda}
 \newcommand{\la}{\lambda}

\begin{document}

\title[$\cD$-elliptic sheaves and odd Jacobians]
{$\cD$-elliptic sheaves and odd Jacobians}

\author{Mihran Papikian}

\address{Department of Mathematics, Pennsylvania State University, University Park, PA 16802}

\email{papikian@math.psu.edu}

\thanks{The author was supported in part by NSF grant DMS-0801208.}



\begin{abstract}
We examine the existence of rational divisors on 
modular curves of $\cD$-elliptic sheaves and on Atkin-Lehner quotients of these curves 
over local fields. Using a criterion of Poonen and Stoll, we show that in infinitely many 
cases the Tate-Shafarevich groups of the Jacobians of these Atkin-Lehner 
quotients have non-square orders. 
\end{abstract}


\maketitle


\section{Introduction} 

Let $F$ is a global field. Let $C$ be a smooth projective geometrically
irreducible curve of genus $g$ over $F$. Denote by $|F|$ the set 
of places of $F$. For $x\in |F|$, denote by $F_x$ the completion of $F$ at $x$. 
A place $x\in |F|$ is called \textit{deficient} for $C$ if $C_{F_x}:=C\times_F F_x$ 
has no $F_x$-rational divisors of  degree $g-1$, cf. \cite{PS}. 
It is known that the number of deficient places is finite. Let $J$ be the Jacobian variety of $C$. Assume 
the Tate-Shafarevich group $\Sh(J)$ is finite.
In \cite{PS}, Poonen and Stoll show that the order of $\Sh(J)$ can be a square as 
well as twice a square. In the first case $J$ is called \textit{even}, and in the second case $J$ is called \textit{odd}. 
The parity of the number of deficient places is directly related to the parity of $J$ \cite[Section 8]{PS}:

\begin{thm}\label{thmPS}
 $J$ is even if and only if the number of deficient places for $C$ is even. 
\end{thm}

Using this theorem, Poonen and Stoll show that infinitely many hyperelliptic 
Jacobians over $\Q$ are odd 
for every even genus. Moreover, for   
certain explicit genus $2$ and $3$ curves over $\Q$ they are able to prove that $\Sh(J)$ is finite 
and has non-square order. In \cite{JL3}, applying Theorem \ref{thmPS} to quotients of Shimura curves 
under the action of Atkin-Lehner involutions, 
Jordan and Livn\'e show that infinitely many of these quotient curves have odd Jacobians. 

For function fields, Proposition 30 in \cite{PS} gives the following example: 
Let $J$ be the Jacobian of the genus $2$ curve 
$$
C: y^2=Tx^6+x-aT
$$
over $\F_q(T)$, where $q$ is odd, and $a\in \F_q^\times$ is a non-square. As one checks, only the 
place $\infty=1/T$ is deficient for $C$. Next, as is observed in \cite[p. 1141]{PS}, 
$C$ defines a rational surface over $\F_q$, so the Brauer group of that surface is finite. The main 
theorem in \cite{G-A} then implies that $\Sh(J)$ is also finite. 
Overall, $\Sh(J)$ is finite and has non-square order.  

In this paper we adapt the idea of Jordan and Livn\'e \cite{JL3} to $\F_q(T)$, and 
exhibit infinitely many curves over $\F_q(T)$ 
whose Jacobians are odd. These curves are obtained as quotients of modular curves of 
$\cD$-elliptic sheaves under the action of Atkin-Lehner involutions. We also show that 
only finitely many of these curves can be hyperelliptic, and in some cases prove that 
the Tate-Shafarevich groups in question are indeed finite. 


\section{Notation and terminology} 

\subsection{Notation} Let $F= \F_q(T)$ be the 
field of rational functions on the projective line $\p:=\p^1_{\F_q}$ over the finite field 
$\F_q$. Fix the place $\infty:=1/T$.  For $x\in |F|$, 
denote by $F_x$ and $\cO_x$ the completion of $F$ and $\cO_{\p,x}$ at $x$, respectively. 
The residue field of $\cO_x$ is denoted by $\F_x$ and the cardinality of $\F_x$ is denoted by 
$q_x$. The degree of $x$ is $\deg(x):=[\F_x:\F_q]$. Let $\varpi_x$ be a uniformizer of $\cO_x$. 
We assume that the valuation $\ord_x:F_x\to \Z$ is normalized by $\ord_x(\varpi_x)=1$. 
Let $A=\F_q[T]$ be the polynomial ring over $\F_q$; 
this is the subring of $F$ consisting of functions which are regular away from $\infty$. For 
a place $x\neq \infty$, 
let $\fp_x$ be the corresponding prime ideal of $A$, and $\wp_x\in A$ be the monic generator 
of $\fp_x$. 

For a ring $H$ with a unit element, we denote by $H^\times$ the group of its invertible elements.  

For $S\subset |F|$, put
$$
\Odd(S)=\left\{
         \begin{array}{ll}
           1, & \hbox{if all places in $S$ have odd degrees;} \\
           0, & \hbox{otherwise.}
         \end{array}
       \right.
$$

\subsection{Quaternion algebras} Let $D$ be a \textit{quaternion
algebra} over $F$, i.e., a $4$-dimensional $F$-algebra with center
$F$ which does not possess non-trivial two-sided ideals.  Denote $D_x:=D\otimes_F F_x$; 
this is a quaternion algebra over $F_x$. By Wedderburn's theorem
\cite[(7.4)]{Reiner}, a quaternion algebra is either a division algebra or 
is isomorphic to the algebra of $2\times 2$ matrices. 
We say that the algebra $D$ \textit{ramifies} (resp. \textit{splits}) 
at $x\in |F|$ if $D_x$ is a division algebra (resp. $D_x\cong \M_2(F_x)$).  
Let $R\subset |F|$ be the set of places where $D$ ramifies. It is known
that $R$ is a finite set of even cardinality, and for any choice of
a finite set $R\subset |F|$ of even cardinality there is a unique,
up to isomorphism, quaternion algebra ramified exactly at the places
in $R$; see \cite[p. 74]{Vigneras}. In particular, $D\cong \M_2(F)$
if and only if $R=\emptyset$. 
We denote the \textit{reduced norm} of $\alpha\in D$ by $\Nr(\alpha)$; for the 
definition see \cite[(9.6a)]{Reiner}. 

An \textit{$\cO_\p$-order} in $D$ is a sheaf of $\cO_\p$-algebras with generic fibre $D$ 
which is coherent and locally free as an $\cO_\p$-module. 
A \textit{$\cD$-bimodule} for an $\cO_\p$-order  $\cD$ in $D$ is an $\cO_\p$-module $\cI$ 
with left and right $\cD$-actions compatible with the $\cO_\p$-action and such that
$$
(\la i)\mu=\la(i\mu), \quad \text{for any }\la, \mu\in \cD \text{ and } i\in \cI. 
$$  
A $\cD$-bimodule $\cI$ is \textit{invertible} if there is another  $\cD$-bimodule $\cJ$ such that 
there are isomorphism of $\cD$-bimodules
$$
\cI\otimes_{\cD}\cJ\cong \cD, \quad \cJ\otimes_{\cD}\cI\cong \cD.
$$
The group of isomorphism classes of invertible $\cD$-bimodules will be denoted by $\Pic(\cD)$: the 
group operation  is $\cI_1\otimes_\cD\cI_2$, cf. 
\cite[(37.5)]{Reiner}. 


\subsection{Graphs}\label{SecG} We recall some of the terminology related to graphs, as presented in
\cite{SerreT} and \cite{Kurihara}. A graph $\cG$ consists of a set
of \textit{vertices} $\Ver(\cG)$ and a set of \textit{edges}
$\Ed(\cG)$. Every edge $y$ has \textit{origin} $o(y)\in \Ver(\cG)$,
\textit{terminus} $t(y)\in \Ver(\cG)$, and \textit{inverse} edge
$\bar{y}\in \Ed(\cG)$ such that $\bar{\bar{y}}=y$ and $o(y)=t(\bar{y})$, $t(y)=o(\bar{y})$.
The vertices $o(y)$ and
$t(y)$ are the \textit{extremities} of $y$. Note that it is allowed
for distinct edges $y\neq z$ to have $o(y)=o(z)$ and $t(y)=t(z)$. 
We say that two vertices are
\textit{adjacent} if they are the extremities of some edge. 
The graph $\cG$ is a \textit{graph with lengths} if we are given a map
$$
\ell=\ell_\cG: \Ed(\cG)\to \N=\{1,2,3, \cdots\}
$$
such that $\ell(y)=\ell(\bar{y})$.  An \textit{automorphism} of
$\cG$ is a pair $\phi=(\phi_1, \phi_2)$ of bijections $\phi_1: \Ver(\cG)\to
\Ver(\cG)$ and $\phi_2: \Ed(\cG)\to \Ed(\cG)$ such that $\phi_1(o(y))=o(\phi_2(y))$,
$\overline{\phi_2(y)}=\phi_2(\bar{y})$, and
$\ell(y)=\ell(\phi_2(y))$.

Let $\G$ be a group acting on a graph $\cG$ (i.e., $\G$ acts via
automorphisms). For $v\in \Ver(\cG)$, denote
$$
\Stab_\G(v)=\{\gamma\in \G\ |\ \gamma v=v\}
$$
the stabilizer of $v$ in $\G$. Similarly, let $\Stab_\G(y)=\Stab_\G(\bar{y})$
be the stabilizer of $y\in \Ed(\cG)$. There is a quotient
graph $\G\bs\cG$ such that $\Ver(\G\bs \cG)=\G\bs \Ver(\cG)$ and
$\Ed(\G\bs \cG)=\G\bs \Ed(\cG)$.


\subsection{Mumford uniformization}\label{ssAC} Let $\cO$ be a complete discrete valuation
ring with fraction field $K$, finite residue field $k$ and a
uniformizer $\pi$. Let $\G$ be a subgroup of $\GL_2(K)$ whose image $\overline{\G}$ in $\PGL_2(K)$ 
is discrete with compact quotient. There is a formal
scheme $\widehat{\Omega}$ over $\Spf(\cO)$ 
which is equipped with a natural action of $\PGL_2(K)$ and parametrizes certain formal 
groups. Kurihara in \cite{Kurihara} extended Mumford's fundamental result \cite{Mumford} and proved the 
following: there is a normal, proper and flat 
scheme $X^\G$ over 
$\Spec(\cO)$ such that the formal completion of $X^\G$ along its closed fibre is isomorphic 
to the quotient $\G\bs \widehat{\Omega}$. The generic 
fibre $X^\G_K$ is a smooth, geometrically integral curve over $K$. The closed fibre 
$X^\G_k$ is reduced with normal crossing
singularities, and every irreducible component is isomorphic to
$\p^1_k$. If $x$ is a double point on $X^\G_k$, then there
exists a unique integer $m_x$ for which the completion of
$\cO_{x,X}\otimes_\cO \widehat{\cO}^\un$ is isomorphic to the
completion of $\widehat{\cO}^\un[t,s]/(ts-\pi^{m_x})$. Here $\widehat{\cO}^\un$ 
denotes the completion of the maximal unramified extension of $\cO$.

\begin{rem} $\widehat{\Omega}$ is the formal scheme associated to Drinfeld's
non-archimedean half-plane $\Omega=\p^{1,\an}_K-\p^{1,\an}_K(K)$
over $K$. For the description of the rigid-analytic structure of
$\Omega$ and the construction of $\widehat{\Omega}$ we refer to Chapter I in 
\cite{BC}. 
\end{rem}

The \textit{dual graph} $\cG$ of $X^\G$ is the following graph
with lengths. The vertices of $\cG$ are the irreducible components
of $X^\G_k$. The edges of $\cG$, ignoring the orientation, are the
singular points of $X^\G_k$. If $x$ is a double point and $\{y,
\bar{y}\}$ is the corresponding edge of $\cG$, then the extremities
of $y$ and $\bar{y}$ are the irreducible components passing through
$x$; choosing between $y$ or $\bar{y}$ corresponds to choosing one
of the branches through $x$. Finally, $\ell(y)=\ell(\bar{y})=m_x$.

Let $\cT$ be the graph whose vertices $\Ver(\cT)=\{[\La]\}$ are the
homothety classes of $\cO$-lattices in $K^2$, and two vertices $[\La]$ and
$[\La']$ are adjacent if we can choose representatives $L\in [\La]$
and $L'\in [\La']$ such that $L'\subset L$ and $L/L'\cong k$. One
shows that $\cT$ is an infinite tree in which every vertex is adjacent to exactly $\# k+ 1$ 
other vertices. This is the \textit{Bruhat-Tits tree} of
$\PGL_2(K)$, cf. \cite[p. 70]{SerreT}. The group $\GL_2(K)$ acts on $\cT$ as the group of linear
automorphisms of $K^2$, so the group $\G$ also acts on $\cT$. 
We assign lengths to the edges of the quotient graph
$\G\bs \cT$: for $y\in \Ed(\G\bs \cT)$ let
$\ell(y)=\#\Stab_{\overline{\G}}(\tilde{y})$, where $\tilde{y}$ is a preimage of $y$
in $\cT$.
By Proposition 3.2 in \cite{Kurihara}, there is an isomorphism $\cG\cong \G\bs
\cT$ of graphs with lengths. 

\begin{notn} For $x\in |F|$, we denote Mumford's formal scheme 
over $\Spf(\cO_x)$ by $\widehat{\Omega}_x$, and the Bruhat-Tits tree of $\PGL_2(F_x)$ 
by $\cT_x$. 
\end{notn}


\section{Modular curves of $\cD$-elliptic sheaves} 

\subsection{$\cD$-elliptic sheaves}\label{ssDES}
The notion of $\cD$-elliptic sheaves was introduced in \cite{LRS}. Here we follow 
\cite{Spiess}, which gives a somewhat different (but equivalent) 
definition of  $\cD$-elliptic sheaves that is more convenient for our purposes. 

From now on we assume that $D$ is a division quaternion algebra which is split at $\infty$. 
Let $\cD$ be an $\cO_\p$-order in $D$ 
such that $\cD_x:=\cD\otimes_{\cO_\p}\cO_x$ is a maximal order in $D_x$
for any $x\neq \infty$, and $\cD_\infty$ is isomorphic to the subring of $\M_2(\cO_\infty)$ 
which are upper triangular modulo $\varpi_\infty$. Let $\cD(-\frac{1}{2}\infty)$ 
denote the two-sided ideal in $\cD$ given by $\cD(-\frac{1}{2}\infty)_x=\cD_x$ for all 
$x\neq \infty$, and $\cD(-\frac{1}{2}\infty)_\infty$ is the radical of $\cD_x$. Concretely, 
 $\cD(-\frac{1}{2}\infty)_\infty$ is the ideal of $\cD_\infty$ 
consisting of matrices which are upper triangular modulo $\varpi_\infty$ with zeros on the diagonal. 

\begin{defn}
A \textit{$\cD$-elliptic sheaf with pole $\infty$} 
over an $\F_q$-scheme $S$ is a pair $E=(\cE, t)$ consisting 
of a locally free right $\cD\boxtimes \cO_S$-module of rank $1$ and an injective homomorphism 
of $\cD\boxtimes \cO_S$-modules 
$$
t:(\mathrm{id}_\p\times \Frob_S)^\ast(\cE\otimes_\cD \cD(-\frac{1}{2}\infty))\to \cE
$$
such that the cokernel of $t$ is supported on the graph $\G_z\subset \p\times_{\Spec \F_q}S$ 
of a morphism $z:S\to \p$ and is a locally free $\cO_S$-module of rank $2$. 
\end{defn}

\begin{rem}
In \cite{Spiess}, this definition is given for an arbitrary central simple algebra $D$ over an 
arbitrary function field $F$. Moreover, the order $\cD$ is only assumed to be hereditary. 
\end{rem}

\begin{thm}\label{thmDES}
The moduli stack of $\cD$-elliptic sheaves of fixed degree $\deg(\cE)=-1$ admits a coarse 
moduli scheme $X^R$. The canonical morphism $X^R\to \p$ is projective with geometrically 
irreducible fibres 
of pure relative dimension $1$, and it is smooth over 
$\p-R-\infty$. 
\end{thm}
\begin{proof}
This theorem follows from one of the main results in \cite{LRS}; 
cf. \cite[$\S$4.3]{Spiess}. 
\end{proof}

The genus of the curve $X^R_F$ is given by the formula (see \cite{PapGenus})
\begin{equation}\label{eqGenus}
g(X^R)=1+\frac{1}{q^2-1}\prod_{x\in R}(q_x-1)-\frac{q}{q+1}\cdot
2^{\# R-1}\cdot \Odd(R).
\end{equation}

\subsection{Atkin-Lehner involutions}
Let $\fP_x$ be the radical of $\cD_x$. 
By \cite[(39.1)]{Reiner}, $\fP_x$ is a two-sided ideal in $\cD_x$, and 
every two-sided ideal of $\cD_x$ is an integral power of $\fP_x$. 
It is known that there exists $\Pi_x\in \fP_x$ such that $\Pi_x\cD_x=\cD_x\Pi_x=\fP_x$. 
The 
positive integer $e_x$ such that $\fP_x^{e_x}=\varpi_x\cD_x$ is the \textit{index} 
of $\cD_x$. With this definition, $e_x=2$ if $x\in R\cup \infty$, and $e_x=1$, otherwise. 
Define the group of divisors 
$$
\Div(\cD):=\left\{\sum_{x\in |F|}n_xx\in \bigoplus_{x\in |F|}\Q x\ \bigg|\
e_xn_x\in \Z\text{ for any }x\in |F|\right\}.
$$
For a divisor $Z=\sum_{x\in |F|}n_xx\in \Div(\cD)$, let 
$\cD(Z)$ be the invertible $\cD$-bimodule given by
$\cD(Z)|_{\p-\Supp(Z)}=\cD|_{\p-\Supp(Z)}$ and $\cD(Z)_x=\fP_x^{-n_xe_x}$ for
all $x\in \Supp(Z)$. For each $f\in F^\times$ there is an associated divisor 
$\mathrm{div}(f)=\sum_{x\in |F|}\ord_x(f)x$, which we consider as an element of $\Div(\cD)$. 
It follows from \cite[(40.9)]{Reiner} that the sequence 
\begin{equation}\label{eqDivD}
0\to F^\times/\F_q^\times\xrightarrow{\mathrm{div}}\Div(\cD)\xrightarrow{Z\mapsto
\cD(Z)}\Pic(\cD)\to 0
\end{equation}
is exact, cf. \cite[$\S$3.2]{Spiess}. Let $\Div^0(\cD)\subset \Div(\cD)$ be the subgroup of degree 
$0$ divisors: $\sum_{x\in |F|}n_xx\in \Div^0(\cD)$ if $\sum_{x\in |F|}n_x\deg(x)=0$. 
Define $\Pic^0(\cD)$ to be the image of $\Div^0(\cD)$ in $\Pic(\cD)$. 
It is easy to check that $\Pic^0(\cD)\cong (\Z/2\Z)^{\# R}$, and is generated by the divisors 
$\left(\frac{\deg(x)}{2}\infty - \frac{1}{2}x\right)$, $x\in R$. 

If $\cL\in \Pic(\cD)$, then
$$
E=(\cE, t)\mapsto E\otimes \cL:=(\cE\otimes_\cD \cL,
t\otimes_\cD \mathrm{id}_\cL)
$$
defines an automorphism of the stack of $\cD$-elliptic sheaves. Moreover, if $\cL\in \Pic^0(\cD)$, then 
this action preserves the substack consisting of $(\cE, t)$ with $\deg(\cE)$ fixed, cf. \cite[$\S$4.1]{Spiess}. 
Hence $W:=\Pic^0(\cD)$ acts on $X^R$ by automorphisms. 

\begin{defn}We call the subgroup $W$ of $\Aut(X^R)$ the 
\textit{group of Atkin-Lehner 
involutions}, and we denote by $w_x\in W$, $x\in R$, 
the automorphism induced  by $\cD\left(\frac{\deg(x)}{2}\infty - \frac{1}{2}x\right)$. 
\end{defn}

\begin{rem}
It follows from \cite[Thm. 4.6]{PapHE} that if $\Odd(R)=0$, then $\Aut(X^R)=W$. 
\end{rem}

The \textit{normalizer} of $\cD_x$ in $D_x$ is the subgroup of $D_x^\times$ 
$$
N(\cD_x)=\{g\in D_x^\times\ |\ g\cD_x g^{-1}=\cD_x\}. 
$$
If $g\in N(\cD_x)$, then $g\cD_x$ is a two-sided ideal of $\cD_x$, so there 
exists $m\in \Z$ such that $g\cD_x=\fP_x^m$
Define $v_{\cD_x}(g)=\frac{m}{e_x}$. Note that 
for $g\in F_x\subset N(\cD_x)$, we have $\ord_x(g)=v_{\cD_x}(g)$. 

Let $
\cC(\cD):=\prod_{x\in |F|}'N(\cD_x)/F^\times\prod_{x\in
|F|}\cD_x^\times$, 
where $\prod_{x\in |F|}'N(\cD_x)$ denotes the
restricted direct product of the groups $\{N(\cD_x)\}_{x\in |F|}$
with respect to $\{\cD_x^\times\}_{x\in |F|}$. Given $a=\{a_x\}_x\in
\prod_{x\in |F|}'N(\cD_x)$, we put $\mathrm{div}(a)=\sum_{x\in
|F|}v_{\cD_x}(a_x)x$. The assignment $a\mapsto \cD(\mathrm{div}(a))$
induces an isomorphism \cite[Cor. 3.4]{Spiess}:
\begin{equation}\label{eq2.1}
\cC(\cD)\cong \Pic(\cD).
\end{equation}

Let $\cD^\infty:=H^0(\p-\infty, \cD)$; this is a maximal $A$-order in $D$. Let 
$\G^\infty:=(\cD^\infty)^\times$ be the units in $\cD^\infty$. Define the \textit{normalizer} of $\cD^\infty$ in $D$ as 
$$
N(\cD^\infty):=\{g\in D^\times\ |\ g\cD^\infty g^{-1}=\cD^\infty \}.
$$
Denote $\cC(\cD^\infty)=N(\cD^\infty)/F^\times\G^\infty$. Then (\ref{eq2.1}) induces an isomorphism 
\begin{equation}\label{eq2.2}
\cC(\cD^\infty)\cong \Pic^0(\cD). 
\end{equation}
By (37.25) and (37.28) in \cite{Reiner}, the natural homomorphism 
\begin{equation}\label{eqI}
N(\cD^\infty)/F^\times\G^\infty \to \prod_{x\in |F|-\infty} N(\cD_x)/F_x^\times \cD_x^\times
\end{equation}
is an isomorphism. Next, by (37.26) and (37.27) in \cite{Reiner}, 
$$
N(\cD_x)/F_x^\times \cD_x^\times\cong\left\{
         \begin{array}{ll}
           1, & \hbox{if $x\not \in R\cup \infty$;} \\
           \Z/2\Z, & \hbox{if $x\in R$.}
         \end{array}
       \right.
$$
For $x\in R$, the non-trivial element of $N(\cD_x)/F_x^\times \cD_x^\times$ 
is the image of $\Pi_x$. According to \cite[(34.8)]{Reiner}, there exist elements 
$\{\la_x\in \cD^\infty\}_{x\in R}$ such that $\Nr(\la_x)A=\fp_x$. The image of $\la_x$ in $\cD_x$ 
can be taken as $\Pi_x$. Overall, $\cC(\cD^\infty)\cong (\Z/2Z)^{\# R}$ is generated by 
$\la_x$'s, and the isomorphism (\ref{eq2.2}) is given by $w_x\mapsto \la_x$. 


\subsection{Uniformization theorems} 
Since $D_\infty\cong \M_2(\Fi)$, the group $\G^\infty$ can be considered as a discrete cocompact 
subgroup of $\GL_2(\Fi)$ via an embedding 
$$
\G^\infty\hookrightarrow D^\times(\Fi)\cong \GL_2(\Fi).
$$
Let $\widehat{X}^R_{\cO_\infty}$ denote 
the completion of $X^R_{\cO_\infty}$ along its special fibre. 
By a theorem of Blum and Stuhler \cite[Thm. 4.4.11]{BS}, there an isomorphism of 
formal $\cO_\infty$-schemes 
\begin{equation}\label{eqBS}
\G^\infty\bs\widehat{\Omega}_\infty\cong \widehat{X}^R_{\cO_\infty}, 
\end{equation}
which is compatible with the action of $W$; see \cite[$\S$4.6]{Spiess}. More precisely, 
the action of $w_x$ on 
$\G^\infty\bs\widehat{\Omega}_\infty$ induced by (\ref{eqBS}) is given by the action 
of $\la_x$ considered as an element 
of $\GL_2(\Fi)$. Note that $\la_x$ is in the normalizer of $\G^\infty$, so 
it acts on the quotient $\G^\infty\bs\widehat{\Omega}_\infty$ and this action 
does not depend on a particular choice of $\la_x$. 

\vspace{0.1in}

Now fix some $x\in R$. Let $\bD$ be the quaternion algebra over $F$ which is ramified
exactly at $(R-x)\cup \infty$. Fix a maximal $A$-order $\fD$ in
$\bD(F)$, and denote
\begin{align*}
&A^x=A[\wp_x^{-1}];\\
&\fD^x=\fD\otimes_A A^x;\\
&\fD^{x,2}=\{\gamma\in \fD^x\ |\ \ord_x(\Nr(\gamma))\in 2\Z\};\\
&\G^x=(\fD^{x,2})^\times.
\end{align*}
If we fix an identification of $\bD_x$ with $\M_2(F_x)$, then 
$\G^x$ is a subgroup of $\GL_2(F_x)$ whose image $\G^x/(A^x)^\times$ in $\PGL_2(F_x)$ 
is discrete and cocompact. Let $\cO_x^{(2)}$ 
be the unramified quadratic extension of $\cO_x$. 
Let $\gamma_x\in \fD^x$ be an element such that $\Nr(\gamma_x)A=\fp_x$. Such 
$\gamma_x$ exists by \cite[(34.8)]{Reiner} and it normalizes $\G^x$, hence acts on  
$\G^x\bs \widehat{\Omega}_x$. Let $\widehat{X}^R_{\cO_x}$ denote 
the completion of $X^R_{\cO_x}$ along its special fibre.
By the analogue of the Cherednik-Drinfeld uniformization, proven in this context by Hausberger 
\cite{Hausberger}, there is an isomorphism of formal $\cO_x$-schemes 
\begin{equation}\label{eqHaus}
[(\G^x\bs \widehat{\Omega}_x)\otimes \cO_x^{(2)}]/(\gamma_x\otimes \Frob_x^{-1})\cong \widehat{X}^R_{\cO_x},
\end{equation}
where $\Frob_x:\cO_x^{(2)}\to \cO_x^{(2)}$ 
denotes the lift of the Frobenius homomorphism $a\mapsto a^{q_x}$ on 
$\overline{\F}_x$ to an $\cO_x$-homomorphism. 

Let $N(\fD^{x,2})$ be the normalizer of $\fD^{x,2}$ in $\bD$, and 
$$\cC(\fD^{x,2}):=N(\fD^{x,2})/F^\times \G^x.$$ 
As in (\ref{eqI}), the natural homomorphism 
$$
N(\fD^{x,2})/F^\times \G^x\to \prod_{y\in |F|-\infty} N(\fD^{x,2}_y)/F_y^\times (\fD^{x,2}_y)^\times
$$
is an isomorphism. 
The normalizer $N(\fD^{x,2}_x)$ is $F_x^\times (\fD^{x}_x)^\times$, so we have 
$$N(\fD^{x,2}_x)/F_x^\times (\fD^{x,2}_x)^\times\cong \Z/2\Z,$$ generated by $\gamma_x$. 
On the other hand, if $y\neq x$, then 
$$
N(\fD^{x,2}_y)/F_y^\times (\fD^{x,2}_y)^\times\cong N(\cD_y)/F_y^\times \cD_y^\times. 
$$
We see that 
$$
\cC(\fD^{x,2})\cong (\Z/2\Z)^{\# R},
$$
 generated by a set of elements $\{\gamma_y\in \fD^{x}\}_{y\in R}$ such that 
 $\Nr(\gamma_y)A=\fp_y$. The group $W$ is canonically isomorphic with 
 $\cC(\fD^{x,2})$ via $w_y\mapsto \gamma_y$. The isomorphism (\ref{eqHaus}) is 
 compatible with the action of $W$: for $y\in R$, the action of $w_y$ on the left hand-side 
 of (\ref{eqHaus})  is given by $\gamma_y$; see \cite[$\S$4.6]{Spiess}.
 
 
 \section{Main results} 
 
\begin{prop}\label{thm3.1} Denote by $\Div^d_{F_x}(X^R)$ the set of 
 Weil divisors on $X^R_{F_x}$ which are rational over $F_x$ and have degree $d$. 
\begin{enumerate}
\item If $x\not \in R$, then $\Div^d_{F_x}(X^R)\neq \emptyset$ for any $d$. 
\item If $x\in R$, then $\Div^d_{F_x}(X^R)\neq \emptyset$ for even $d$, 
and $\Div_{F_v}^d(X^R)=\emptyset$ for odd $d$. 
\end{enumerate}
 \end{prop}  
 \begin{proof} For $n\geq 1$, denote by $\F_x^{(n)}$ the degree $n$ extension 
 of $\F_x$, and by $F_x^{(n)}$ the degree $n$ unramified extension 
 of $F_x$. 
 
 First, suppose $x\not\in R\cup \infty$. By Theorem \ref{thmDES}, $X^R_{\F_x}$ 
 is a smooth projective curve. 
 Weil's bound on the number of rational points on a curve over a finite field 
 guarantees the existence of an integer $N\geq 1$ such that 
 $X^R_{\F_x}(\F_x^{(n)})\neq \emptyset$ for any $n\geq N$. The geometric version of 
 Hensel's lemma \cite[Lem. 1.1]{JL} implies that $X^R_{F_x}(F_x^{(n)})\neq \emptyset$. Let 
 $P\in X^R_{F_x}(F_x^{(N+1)})$ and $Q\in X^R_{F_x}(F_x^{(N)})$. The 
 divisor $d\cdot Z$, where 
 $$
 Z=\sum_{\sigma\in \Gal(F_x^{(N+1)}/F_x)}P^\sigma- 
 \sum_{\tau\in \Gal(F_x^{(N)}/F_x)}Q^\tau,
 $$ 
 is $F_x$-rational and has degree $d$. 
 
Next, suppose $x=\infty$. By (\ref{eqBS}), $X^R_{\Fi}$ is Mumford uniformizable. This implies 
 that $X^R_{\Fi}$ has a regular model over $\cO_\infty$ whose special fibre 
 consists of $\F_\infty$-rational $\p^1$'s crossing at $\F_\infty$-rational points. In particular, 
 over any extension $\F_\infty^{(n)}$, $n\geq 2$, there are smooth 
 $\F_\infty^{(n)}$-rational points. Again by Hensel's lemma 
 \cite[Lem. 1.1]{JL}, there are $\Fi^{(n)}$-rational points on $X^R_{\Fi}$ for any $n\geq 2$. 
 The trace to $\Fi$ of such a point is in $\Div^n_{\Fi}(X^R)$. One obtains a rational divisor 
 of degree $1$ by taking the difference of degree $3$ and  $2$ rational divisors.  This proves (1). 
 
 Finally, suppose $x\in R$. By \cite[Thm. 4.1]{PapLocProp}, $X^R_{F_x}(F_x^{(2)})\neq \emptyset$. 
 Taking the trace of an $F_x^{(2)}$-rational point and multiplying the resulting divisor by $n$, we see that 
 $\Div^{2n}_{F_x}(X^R)\neq \emptyset$ for any $n$. 
Now suppose $d$ is odd but $\Div^d_{F_x}(X^R)\neq
\emptyset$. Let $Z\in \Div^d_{F_x}(X^R)$. Write $Z=Z_1-Z_2$, where $Z_1$
and $Z_2$ are effective divisors. Since
$\deg(Z)=\deg(Z_1)-\deg(Z_2)$ is odd, exactly one of these divisors
has odd degree. Denote by $F_x^\alg$ the algebraic closure of $F_x$, 
$F_x^\sep$ the separable closure of $F_x$, and let $G:=\Gal(F_x^\sep/F_x)$. 
Since $Z$ is $F_x$-rational, both $Z_1$ and $Z_2$ are $G$-invariant.
Assume without loss of generality that $\deg(Z_1)$ is odd. Write
$Z_1=Z_o+Z_e$, where 
$Z_o=\sum_{P\in X^R_{F_x}(F_x^\alg)} n_P P$, $n_P\in \Z$ are odd, and 
$Z_e=\sum_{Q\in X^R_{F_x}(F_x^\alg)} n_Q Q$, $n_Q\in \Z$ are even. 
Again $Z_o$ and $Z_e$ are $G$-invariant. Since
$\deg(Z_e)$ is even, $Z_o$ is non-zero. Since $\deg(Z_o)$ is
necessarily odd, the support of $Z_o$ must consist of an odd number
of points. This set of points is $G$-invariant. We have a finite set
of odd cardinality with an action of $G$, so one of the orbits
necessarily has odd length. Thus, there is a
point $P$ in the support of $Z$ such that the set of Galois conjugates of 
$P$ has odd cardinality. This implies that the separable degree 
$[F_x(P):F_x]_s$ is odd. If $P$ is not separable, then the degree of inseparability 
of $F_x(P)$ over $F_x$ divides the weight $n_P$ of $P$ in $Z$ (as $Z$ is 
$F_x$-rational). Since $n_P$ is odd by assumption, the inseparable 
degree $[F_x(P):F_x]_i$ is also odd. Overall, the degree of the 
extension $F_x(P)/F_x$ is odd. We conclude that there is a finite 
extension $K/F_x$ of odd degree such that $X^R_{F_x}(K)\neq \emptyset$. 
This contradicts \cite[Thm. 4.1]{PapLocProp}, so $\Div^d_{F_x}(X^R)$ 
must be empty. 
 \end{proof}
 
 \begin{thm}
 Consider the following conditions: 
  \begin{enumerate}
 \item $q$ is even;
 \item $q$ is odd, $\# R=2$, and $\Odd(R)=1$.
 \end{enumerate}
 If one of these conditions holds, then the deficient places for $X^R$ are 
 the places in $R$. Otherwise, there are no deficient places for $X^R$. 
 In either case, by Theorem \ref{thmPS}, the Jacobian variety of $X^R_F$ is even. 
 \end{thm}
 \begin{proof}
 An elementary analysis of (\ref{eqGenus}) shows that the genus $g(X^R)$ 
 is even if and only if one of the above conditions holds. The claim of the 
 theorem then follows from Proposition \ref{thm3.1}.  
 \end{proof}
 
 Next, we examine the existence of rational divisors on the quotients of $X^R$ under the 
 action of Atkin-Lehner involutions. 
 For a fixed $y\in R$ we denote by $X^{(y)}$ the quotient curve $X^R/w_y$. 
 
 \begin{prop}\label{prop3.5} Denote by $\Div^d_{F_x}(X^{(y)})$ the set of 
 Weil divisors on $X^{(y)}_{F_x}$ which are rational over $F_x$ and have degree $d$. 
\begin{enumerate}
\item If $x\not \in R$ or $x=y$, then $\Div^d_{F_x}(X^{(y)})\neq \emptyset$ for any $d$. 
\item If $x\in R-y$ and $d$ is even, then $\Div^d_{F_x}(X^{(y)})\neq \emptyset$.
\item If $x\in R-y$ and $\Div^d_{F_x}(X^{(y)})\neq \emptyset$ for an odd $d$, then 
there is an extension $K/F_x$ of odd degree such that $X^{(y)}_{F_x}(K)\neq \emptyset$. 
\end{enumerate}
 \end{prop}  
 \begin{proof} Since the Atkin-Lehner involutions are defined in terms of the moduli problem, 
 the quotient morphism $\pi:X^R_{F_x}\to X^{(y)}_{F_x}$ is defined over $F_x$. 
 Hence, if $Z\in\Div^d_{F_x}(X^{R})$, then the pushforward $\pi_\ast(Z)$ is in $\Div^d_{F_x}(X^{(y)})$, so 
 Proposition \ref{thm3.1} implies (2) and (1) for $x\not \in R$. Part (3) follows from the argument 
 in the proof of Proposition \ref{thm3.1}. It remains to prove that 
 $\Div^d_{F_y}(X^{(y)})\neq \emptyset$ for any $d$. 
 By (\ref{eqHaus}) and ensuing discussion, $X^{R}_{F_y}$ is the 
 $w_y\otimes \Frob_y^{-1}$ quadratic twist of the Mumford curve $\G^y\bs \widehat{\Omega}_y$. 
 Hence the quotient $X^{(y)}_{F_y}$ of $X^R_{F_y}$ by $w_y$ 
 is Mumford uniformizable (without a twist) and one can argue as in the proof 
 of Proposition \ref{thm3.1} in the case when $x=\infty$. 
 \end{proof}
 
 \begin{prop}\label{prop4.9}
 Assume $q$ is odd, and $x, y\in R$ are two distinct places of even degrees. 
 If $d$ is odd, then $\Div_{F_x}^{(d)}(X^{(y)})=\emptyset$. 
 \end{prop}
 \begin{proof} Suppose $d$ is odd and $\Div_{F_x}^{(d)}(X^{(y)})\neq \emptyset$. Then 
 by Proposition \ref{prop3.5} there is an extension $K/F_x$ of odd degree such that 
 $X^{(y)}_{F_x}(K)\neq \emptyset$. The graph $G:=\G^x\bs \cT_x$ is the dual graph 
of the Mumford curve uniformized by $\G^x$; see $\S$\ref{ssAC}. 
From (\ref{eqHaus}) we get an action of $W$ on $G$. 
 The same argument as in \cite[p. 683]{JL3} shows that if $X^{(y)}_{F_x}(K)\neq \emptyset$, then 
 there is an edge $s$ in $G$ such that the following two conditions hold:
\begin{enumerate}
\item either $\ell(s)$ is even or $w_y(s)=s$;
\item either $w_x(s)=\bar{s}$ or $w_xw_y(s)=\bar{s}$.
\end{enumerate}
By \cite[Lem. 4.4]{PapLocProp}, 
an edge of $G$ has length $1$ or $q+1$, and the number of edges of length $q+1$ is equal to 
$$2^{\# R-1}\Odd(R-x)(1-\Odd(x)).$$ 
 From the assumption that $y$ has even degree we get that all 
 edges of $G$ have length $1$. Thus, for the existence of $K$-rational points on $X^{(y)}_{F_x}$
 we must have $w_y(s)=s$, and either $w_x(s)=\bar{s}$ or
$w_xw_y(s)=\bar{s}$. Obviously $w_y(s)=s$ and
$w_xw_y(s)=\bar{s}$ imply $w_x(s)=\bar{s}$. Therefore, 
the considerations reduce to a single case
$$
w_x(s)=\bar{s}\quad \text{and}\quad w_y(s)=s.
$$
 
 Let $\tilde{s}$ be an edge of $\cT_x$ lying above $s$. Modifying
$\gamma_x$ by an element of $\G^x$, we may assume that
$\gamma_x\tilde{s}=\bar{\tilde{s}}$. Let $v$ be one of the extremities of $\tilde{s}$. Then
$\gamma_x^2$ fixes $v$ and $\Nr(\gamma_x^2)$ generates $\fp_x^2$.
Thus, $\gamma_x^2\in F_x^\times\mu \GL_2(\cO_x)\mu^{-1}$ for some
$\mu\in \GL_2(F_x)$. By the norm condition, we get
$\gamma_x^2=\wp_x c$, where $c\in \mu \GL_2(\cO_x)\mu^{-1}$.
Hence $\ord_x(\Nr(\gamma_x^2/\wp_x))=0$. On the other hand, since
$\gamma_x^2/\wp_x$ also belongs to $\fD^x$, $c$ belongs to a maximal
$A$-order $\fD'$ in $\bD$ (in fact, $\fD'=\mu
\GL_2(\cO_x)\mu^{-1}\cap \fD^x$). Since $\Nr(c)$ has zero valuation
at every $v\in |F|-\infty$, $c\in (\fD')^\times$. By our assumption,
$\deg(y)$ is even and $\bD$ is ramified at $y$ and $\infty$, so
$(\fD')^\times\cong \F_q^\times$; cf. \cite[Lem. 1]{DvG}. Hence $\gamma_x^2=c\wp_x$, where
$c\in \F_q^\times$. Since $\deg(x)$ is even, $c$ must be a
non-square, as otherwise $\infty$ splits in $F(\sqrt{c\wp_x})$, which
contradicts the fact that this is a subfield of the quaternion algebra $\bD$
ramified at $\infty$. Fix a non-square $\xi\in \F_q^\times$. Overall, 
we conclude that the condition $w_x(s)=\bar{s}$ translates into 
$$
\quad \gamma_x^2=\xi \wp_x, 
$$
for an appropriate choice of $\gamma_x$. 
 
Modifying $\gamma_y$ by an element of $\G^x$, we can further assume that 
$\gamma_y(\tilde{s})=\tilde{s}$.  Next, note that $\gamma_y$ belongs to some maximal $A$-order $\fD''$ in
$\bD$. Since $\bD$ is ramified at $y$ and $\Nr(\gamma_y)A=\fp_y$, the element 
 $\gamma_y$ generates the radical of $\fD^x_y$. Hence
$\gamma_y^2=c\cdot \wp_y$, where $c\in \fD''$. Comparing the norms of both sides,
we see that $c$ must be a unit in $\fD''$. The same argument 
as with $\fD'$ shows that $(\fD'')^\times\cong \F_q^\times$, so after possibly scaling
$\gamma_y$ by a constant in $\F_q^\times$, we get
$$
\gamma_y^2=\xi \wp_y.
$$

Let $\langle \G^x, \gamma_y \rangle$ 
be the subgroup of $\GL_2(F_x)$ generated by $\G^x$ and $\gamma_y$.  
By construction, the element $\gamma_y$ fixes $\tilde{s}$. 
Since the edges of $G$ have length $1$, 
the stabilizer of $\tilde{s}$ in $\G^x$ is $(A^x)^\times$. Therefore, 
$$
\Stab_{\langle \G^x, \gamma_y \rangle}(\tilde{s})/(A^x)^\times\subset\F_q(\gamma_y)^\times. 
$$
On the other hand, $\gamma_x^{-1}\gamma_y\gamma_x(\tilde{s})=\tilde{s}$. We conclude 
that there is $n\in \Z$ and $a,b\in \F_q$ ($a,b$ are not both zero) such that 
$$
\gamma_y\gamma_x=\wp_x^n\gamma_x(a+b\gamma_y). 
$$
 Now the same argument as in the proof of Part (3) of Theorem 4.1 in \cite{PapLocProp} shows that 
for such an equality to be true we must have $n=0$, $a=0$ and $b=-1$, i.e., 
$$
\gamma_y\gamma_x=-\gamma_x\gamma_y. 
$$

The quadratic extensions $F(\gamma_x)$ and $F(\gamma_y)$ of $F$ are obviously linearly 
disjoint. Therefore, $\bD$ is isomorphic to the quaternion algebra 
$H(\xi\wp_x, \xi\wp_y)$ over $F$ having the presentation:
$$
i^2=\xi\wp_x, \quad j^2=\xi\wp_y, \quad ij=-ji. 
$$
 As is well-known, the algebra $H(\xi\wp_x, \xi\wp_y)$ ramifies (resp. splits) at $v\in |F|$ if 
 and only if the local symbol $(\xi\wp_x, \xi\wp_y)_v=-1$ (resp. $=1$); cf. \cite[p. 32]{Vigneras}. 
 On the other hand, by \cite[p. 210]{SerreLF} 
 $$
 (\xi\wp_x, \xi\wp_y)_x=\left(\frac{\xi\wp_y}{\fp_x}\right) \quad \text{and}\quad 
  (\xi\wp_x, \xi\wp_y)_y=\left(\frac{\xi\wp_x}{\fp_y}\right), 
 $$
 where $\left(\frac{\cdot}{\cdot}\right)$ is the Legendre symbol. 
 Since $x$ and $y$ have even degree, $\xi$ is a square modulo $\fp_x$ and $\fp_y$. Thus, 
$\left(\frac{\xi\wp_y}{\fp_x}\right)=\left(\frac{\wp_y}{\fp_x}\right)$ and 
$\left(\frac{\xi\wp_x}{\fp_y}\right)=\left(\frac{\wp_x}{\fp_y}\right)$. The algebra $\bD$ 
splits at $x$ and ramifies at $y$, so we must have 
$$
\left(\frac{\wp_y}{\fp_x}\right)=1\quad \text{and}\quad 
 \left(\frac{\wp_x}{\fp_y}\right)=-1. 
$$
But the quadratic reciprocity \cite[Thm. 3.5]{Rosen} says that 
$$
\left(\frac{\wp_y}{\fp_x}\right)\left(\frac{\wp_x}{\fp_y}\right)=
(-1)^{\frac{q-1}{2}\deg(x)\deg(y)}=1.
$$
This leads to a contradiction, so $\Div_{F_x}^{d}(X^{(y)})=\emptyset$. 
\end{proof}

\begin{thm}\label{thmMain} Assume $q$ is odd and all places in $R$ have even degrees. 
Consider the following three conditions:
\begin{enumerate}
\item $R=\{x,y\}$, i.e., $\# R=2$;
\item $\left(\frac{\wp_y}{\fp_x}\right)=-1$;
\item $\deg(y)$ is not divisible by $4$.
\end{enumerate}

If one of these conditions fails, then there are no deficient places for  $X^{(y)}$. 
If all three conditions hold, then $x$ is the only deficient place for $X^{(y)}$. In the first case 
the Jacobian of $X^{(y)}_{F}$ is even and in the second case it is odd.
\end{thm}
\begin{proof} Let $\Fix(w_y)$ be the number of fixed points of $w_y$ acting on $X^R_F$. 
 By the Hurwitz genus formula applied to the quotient map $\pi: X^R_F\to X^{(y)}_F$, 
the genus of $X^{(y)}_F$ is equal to 
 $$
 g(X^{(y)})=\frac{g(X^R)+1}{2}-\frac{\Fix(w_y)}{4}.
 $$
 (note that $\pi$ has only tame ramification). On the other hand, by \cite[Prop. 4.12]{PapHE}
$$
\Fix(w_y)=h(\xi\wp_y)\prod_{x\in R}\left(1-\left (\frac{\xi\wp_y}{\fp_x}\right)\right),
$$ 
where $\xi\in \F_q^\times$ is a fixed non-square, 
and $h(\xi\wp_y)$ denotes the ideal class number of the Dedekind ring $\F_q[T, \sqrt{\xi\wp_y}]$. 
(A remark is in order: In \cite{PapHE}, $w_y$ is defined analytically as the involution 
 of $\G^\infty\bs \Omega_\infty$ induced by $\la_y$, hence here we are implicitly 
 using the fact that (\ref{eqBS}) is compatible with the action of $W$.)  Combining these formulas, we get 
$$
g(X^{(y)})=1+\frac{1}{2(q^2-1)}\prod_{x\in R}(q_x-1)-\frac{h(\xi\wp_y)}{4}
\prod_{x\in R}\left(1-\left (\frac{\xi\wp_y}{\fp_x}\right)\right). 
$$ 
It is easy to see that the middle summand is always an even integer. Hence $g(X^{(y)})$ 
is even if and only if the last summand is odd. 
According to \cite[Thm. 1]{Cornelissen}, the 
class number $h(\xi\wp_y)$ is always even and it is divisible by $4$ if and only if $\deg(y)$ is divisible by $4$. 
Using this fact, one easily checks that the last summand is odd 
if and only if the three conditions are satisfied. 
The theorem now follows from Propositions \ref{prop3.5} and \ref{prop4.9}
\end{proof}

There are infinitely many pairs $R=\{x, y\}$ for which 
the conditions in Theorem \ref{thmMain} are satisfied. Indeed, fix an arbitrary $y$ 
such that $\deg(y)\equiv 2\ (\mod\ 4)$. 
By the function field analogue of Dirichlet's theorem \cite[Thm. 4.7]{Rosen}, there are infinitely many 
places $x\in |F|$ of even degree such that $\left(\frac{\wp_x}{\fp_y}\right)=-1$. The quadratic reciprocity 
implies that for such places $\left(\frac{\wp_y}{\fp_x}\right)=-1$. 
Hence there are infinitely many 
$X^{(y)}_F$ with odd Jacobians. 

\begin{rem}
 For a fixed $q$ there are only finitely many $R$ such that $X^{(y)}_F$ is hyperelliptic. 
 To see this, fix some $x\not\in R\cup \infty$. Corollary 4.8 in \cite{PapGenus} gives a lower 
 bound on the number of $\F_x^{(2)}$-rational points on $X^R_{\F_x}$. 
 Since the quotient map $X^R_{\F_x}\to X^{(y)}_{\F_x}$ 
 is defined over $\F_x$ and has degree $2$, from this bound we get  
 $$
 \# X^{(y)}_{\F_x}(\F_x^{(2)})\geq 
 \frac{1}{2}\# X^R_{\F_x}(\F_x^{(2)})\geq \frac{1}{2(q^2-1)}\prod_{z\in R\cup x}(q_z-1).  
 $$
 By \cite[Prop. 5.14]{LK}, if $X^{(y)}_F$ is hyperelliptic, then $X^{(y)}_{\F_x}$ is also 
 hyperelliptic. Hence there is a degree-2 morphism $X^{(y)}_{\F_x} \to \p^1_{\F_x}$ 
 defined over $\F_x$. 
 This implies 
 $$
 \# X^{(y)}_{\F_x}(\F_x^{(2)}) \leq 2 \# \p^1_{\F_x}(\F_x^{(2)}) =2(q_x^2+1). 
 $$
 Comparing with the earlier lower bound on $\# X^{(y)}_{\F_x}(\F_x^{(2)}) $, we get 
 \begin{equation}\label{eqeq}
 \prod_{z\in R\cup x}(q_z-1)\leq 4(q_x^2+1)(q^2-1). 
 \end{equation}
 Let $r=\sum_{z\in R}\deg(z)$. By  \cite[Lem.7.7]{PapLocProp}, 
 we can choose $x\not \in R\cup \infty$ such that 
 $\deg(x)\leq \log_q(r+1)+1$. Since 
 $\prod_{z\in R}(q_z-1)\geq q^{r/2}$, the inequality (\ref{eqeq}) implies 
 $q^{r/2}<32q^3 r$,  which obviously is possible only for finitely many $R$. 
 Therefore, only finitely many $X^{(y)}_F$ are hyperelliptic. 
 \end{rem}

Denote by $J^{(y)}$ the Jacobian variety of $X^{(y)}_F$.  
To conclude the paper, we explain how one can deduce in some cases 
that $\Sh(J^{(y)})$ is finite and has non-square order. (Of course, it is 
expected that Tate-Shafarevich groups are always finite.)

The definitions of the concepts discussed in this paragraph can be found in \cite{GR}. 
Let $\fn\lhd A$ be an ideal. 
Let $X_0(\fn)$ be the compactified Drinfeld modular curve classifying pairs 
$(\phi, C_\fn)$, where $\phi$ is a rank-$2$ Drinfeld 
$A$-modules and $C_\fn\cong A/\fr$ is a cyclic subgroup of $\phi$. Let $J_0(\fn)$ denote 
the Jacobian of $X_0(\fn)_F$. Let $\G_0(\fn)$ be the level-$\fn$ Hecke congruence 
subgroup of $\GL_2(A)$. Let 
$S_0(\fn)$ be the $\C$-vector space 
of automorphic cusp forms 
of Drinfeld type on $\G_0(\fn)$. Let $\T(\fn)$ be the commutative $\Z$-algebra generated 
by the Hecke operators acting on $S_0(\fn)$. The Hecke algebra $\T(\fn)$ is 
a finitely generated free $\Z$-module which also naturally acts on $J_0(\fn)$.  
Let $f\in S_0(\fn)$ be a newform which is an eigenform for all $t\in \T(\fn)$. Denote by $\la_f(t)$ 
the eigenvalue of $t$ acting on $f$. The map  
$\T(\fn)\to \C$, $t\mapsto \la_f(t)$, is an algebra homomorphism; denote its kernel 
by $I_f$. The image $I_f(J_0(\fn))$ is an abelian subvariety of $J_0(\fn)$ defined over $F$. 
Let $A_f:=J_0(\fn)/I_f(J_0(\fn))$. Similar to the case of classical modular Jacobians over $\Q$, 
the Jacobian $J_0(\fn)$ is isogenous 
over $F$ to a direct product of abelian varieties $A_f$, where each $f$ is a newform of some 
level $\fm|\fn$ (a given $A_f$ can appear more than once in the decomposition of $J_0(\fr)$).
This implies that $\Sh(J_0(\fn))$ is finite if and only if $\Sh(A_f)$ is finite for all such $A_f$. 
On the other hand, by the main theorem of  \cite{KT}, $\Sh(A_f)$ is finite 
if and only if 
$$
\ord_{s=1}L(A_f, s)=\mathrm{rank}_\Z A_f(F),
$$
where $L(A_f, s)$ denotes the $L$-function of $A_f$; see \cite{KT} or \cite{Schneider} for 
the definition.  

Let $J^R$ denote the Jacobian of $X^R_F$. Let $\fr:=\prod_{x\in R}\fp_x$. 
The Jacquet-Langlands correspondence over $F$ 
in combination with some other deep results implies that there 
is a surjective homomorphism $J_0(\fr)\to J^R$ defined over $F$; see \cite[Thm. 7.1]{PapJL}. 
Since by construction $X^{(y)}$ is a quotient of $X^R$, there is also a surjective 
homomorphism $J^R\to J^{(y)}$ defined over $F$. Thus, there is a 
surjective homomorphism $J_0(\fr)\to J^{(y)}$ defined over $F$. This implies that 
if $\Sh(J_0(\fr))$ is finite, then $\Sh(J^{(y)})$ is also finite. 

Now assume $q$ is odd, $R=\{x, y\}$, and $\deg(x)=\deg(y)=2$. In this case $J_0(\fr)$ is isogenous 
to $J^R$ as both have dimension $q^2$. The dimension of $J^{(y)}$ is $(q^2-1)/2$. 
 There are no old forms of level $\fr$, since $S_0(1)$, $S_0(\fp_x)$ and $S_0(\fp_y)$ 
 are zero dimensional. Let $f\in S_0(\fr)$ be a Hecke eigenform. 
 The $L$-function $L(f, s)$ of $f$ is a polynomial in $q^{-s}$ of degree 
$\deg(x)+\deg(y)-3=1$, cf. \cite[p. 227]{Tamagawa}. 
Hence $\ord_{s=0}L(f, s)\leq 1$. 
Using the analogue of the Gross-Zagier formula over $F$ \cite[p. 440]{RT}, one concludes 
that $\ord_{s=1}L(A_f, s)\leq \mathrm{rank}_\Z A_f(F)$. The converse inequality is known to hold 
for any abelian variety over $F$; see the main theorem of \cite{Schneider}. 
Hence $\Sh(A_f)$ is finite, which, as we explained, implies that $\Sh(J^{(y)})$ is also finite. 
It remains to
show that one can choose $x$ and $y$ so that the conditions in Theorem \ref{thmMain} hold, 
and therefore $\Sh(J^{(y)})$ is finite and has non-square order. 
We need to show that one can choose 
$x$ and $y$ such that $\deg(x)=\deg(y)=2$ and $\left(\frac{\wp_y}{\fp_x}\right)=-1$. 
Fix some $x\in |F|$ with $\deg(x)=2$. 
Consider the geometric quadratic extension $K:=F(\sqrt{\wp_x})$ of $F$, and let $C$ be the corresponding smooth 
projective curve over $\F_q$. Since $\deg(x)=2$, the genus of this curve is zero, so 
$C\cong \p^1_{\F_q}$. Using this observation, one easily computes that the number 
of places of $F$ of degree $2$ which remain inert in $K$ is
$(q^2-1)/4>0$. 
Thus, we can always choose $y\in |F|$ of degree $2$ such that 
$\left(\frac{\wp_y}{\fp_x}\right)=-1$.  



\end{document}